\documentclass[11pt]{amsart}
\usepackage{eucal}

\title[The Teichm\"uller--Randers metric]{The Teichm\"{u}ller--Randers metric}
\author[Miyachi, Ohshika and Papadopoulos]{Hideki Miyachi, Ken'ichi Ohshika, and Athanase Papadopoulos}
\address{Hideki Miyachi,
School of Mathematics and Physics,
College of Science and Engineering,
Kanazawa University,
Kakuma-machi, Kanazawa,
Ishikawa, 920-1192, Japan}
\email{miyachi@se.kanazawa-u.ac.jp}  
\address{Ken'ichi Ohshika,
Department of Mathematics,
Gakushuin University,
Mejiro, Toshima-ku, Tokyo, Japan}
  \email{ohshika@math.gakushuin.ac.jp}
\address{Athanase Papadopoulos, 
Institut de Recherche Math\'ematique Avanc\'ee (Universit\'e de Strasbourg et CNRS),l((
7 rue Ren\'e Descartes
67084 Strasbourg Cedex France}
 \email{papadop@math.unistra.fr}
\date{\today}

\thanks{This work is supported by JSPS KAKENHI Grant Numbers
16K05202, partially, 16H03933, 17H02843. The authors thank the Institut Henri Poincar\'e for its support during a stay of the three authors there, in November 2022 where this work was completed,  in the setting of the program \emph{Research in Paris}.}

\makeatletter
\@addtoreset{equation}{section}

\makeatother

\usepackage{amsfonts, amsmath, amssymb, amsthm}
\usepackage{url}
\usepackage[dvipdfmx]{graphicx}
\usepackage[dvipdfmx]{color}
\usepackage{amscd}
\usepackage[right]{lineno}
\usepackage[capitalize]{cleveref}
\modulolinenumbers[5]

\usepackage{enumerate}

\newtheorem{theorem}{Theorem}[section]
\newtheorem{lemma}[theorem]{Lemma}
\newtheorem{proposition}[theorem]{Proposition}
\newtheorem{corollary}[theorem]{Corollary}

\newtheorem{claim}{Claim}

\newcommand{\teich}{\mathcal{T}}

\newcommand{\ext}{{\rm Ext}}

\newcommand{\hyperbolic}{\mathbb{H}}

\newcommand{\reals}{\mathbb{R}}

\newcommand{\teichmullernorm}[2]{\kappa(#1;#2)}
\newcommand{\randersteichmullernorm}[3]{\kappa^{#3}(#1;#2)}

\begin{document}
\maketitle

\begin{abstract}
In this paper, we introduce a new asymmetric weak metric on the Teichm\"uller space of a closed orientable surface with (possibly empty) punctures.
This new metric, which we call the Teichm\"{u}ller--Randers metric, is an asymmetric deformation of the Teichm\"uller metric, and is obtained by adding to the infinitesimal form of the Teichm\"uller metric a differential 1-form.  We study basic properties of the  Teichm\"{u}ller--Randers metric. 
In the case when the 1-form is exact, any Teichm\"uller geodesic between two points  is a unique Teichm\"{u}ller--Randers  geodesic between them. 
A particularly interesting case is when the differential 1-form is (up to a factor) the differential of the logarithm of the extremal length function associated with a measured foliation. 
We show that in this case the Teichm\"{u}ller--Randers metric is incomplete in any Teichm\"uller disc, and we give a characterisation of geodesic rays with bounded length in this disc in terms of their directing measured foliations.

\medskip

\noindent {\sc R\'esum\'e}. Dans cet article, nous introduisons une nouvelle m\'etrique asym\'etrique sur l'espace de Teichm\"uller d'une surface ferm\'ee orientable avec ou sans perforations. Cette nouvelle m\'etrique, que nous appelons m\'etrique de Teichm\"{u}ller--Randers, est une d\'eformation asym\'etrique de la m\'etrique de 
Teichm\"uller obtenue en ajoutant \`a la forme infinit\'esimale de cette derni\`ere une forme diff\'erentielle de degr\'e 1. Nous \'etudions les propri\'et\'es de base de la m\'etrique de Teichm\"{u}ller--Randers. Nous d\'emontrons que dans le cas o\`u la forme diff\'erentielle est exacte, toute g\'eod\'esique entre deux points pour la m\'etrique de Teichm\"uller est aussi une g\'eod\'esique pour la m\'etrique de Teichm\"{u}ller--Randers, et que c'est l'unique g\'eod\'esique joignant ces deux points. 
  Un cas particuli\`erement int\'eressant est celui o\`u la forme diff\'erentielle est (\`a un multiple pr\`es) la diff\'erentielle du logarithme de la longueur extr\'emale associ\'ee \`a un feuilletage mesur\'e. Nous montrons que dans ce cas la m\'etrique de  Teichm\"{u}ller--Randers restreinte \`a un disque de Teichm\"uller  quelconque n'est pas compl\`ete et nous caract\'erisons les rayons g\'eod\'esiques de longueur born\'ee dans ce disque en fonction des feuilletages mesur\'es qui les dirigent.

    \end{abstract}
\sloppy
\section{Inroduction}
A Randers metric is a deformation of a Riemannian or Finsler metric obtained by adding to its infinitesimal form a differential 1-form.
In \cite{MOP}, in the case where the surface is a torus, we exhibited a natural family of Randers metrics which connects the Teichm\"uller metric on the Teichm\"uller space of that surface to its Thurston asymmetric metric.  It is natural to study now the same kind of deformation of the Teichm\"uller metric on theTeichm\"uller space $ \teich_{g,m}$ of a general closed orientable surface $\Sigma_{g,m}$ of genus $g$ with $m$ punctures, and this is what we do in the present paper. It turns out that the metrics in this family are interesting to study in this general setting and this is what we propose to show in this paper.

In its original form given in \cite{MR3371}, a Randers metric is associated with an $n$-dimensional Riemannian manifold $(M,g)$ and a 1-form $\omega$ on $M$ satisfying $\|\omega\|_g<1$ at every point of $M$.
In this situation,  the associated \emph{Randers metric} is a Finsler asymmetric metric on $M$ defined infinitesimally by $F(v)=g(v,v)^{1/2}+\omega(v)$. Randers metrics have applications in the physical world, and they been widely studied since their appearance. 
The same construction also works when the original metric is not Riemannian, but Finsler, like in the case we study here.

In this paper we study a Randers deformation of the Teichm\"uller metric $\kappa$ on $\teich_{g,m}$, which we call the \emph{Teichm\"{u}ller--Randers metric} associated with a real 1-form $\omega$, defined using
\begin{equation}
\label{eq:Finsler-Randers}
\randersteichmullernorm{x}{v}{\omega}
=\teichmullernorm{x}{v}+\omega(v)
\end{equation}
In a natural way, the lengths of differentiable arcs on $\teich_{g,m}$ can be defined using this metric, and the distance between two points is set to be the infimum of the lengths of arcs connecting them.
The Teichm\"{u}ller--Randers distance may take negative values for a general 1-form $\omega$, but it gives a Finsler metric when the Teichm\"uller norm $\|\omega\|_T(x)$ of $\omega$ at $x$, (i.e., the supremum of the value of $\omega$ on the tangent vectors at $x$ with Teichm\"{u}ller norm $\leq 1$)  is less than $1$ at every point $x$ of $\teich_{g,m}$. 
The Teichm\"{u}ller--Randers metric becomes a weak metric when the Teichm\"uller norm of $\omega$ is $1$ (see \cref{subsec:weak_metric}).

We have already introduced and studied the Teichm\"{u}ller--Randers metric in the case of torus.
In \cite{BPT}, Belkhirat, Papadopoulos and Troyanov showed that Thurston's asymmetric metric coincides with the weak distance on the upper-half plane $\mathbb{H}$ defined by
\begin{equation}
\label{eq:BPT}
\delta(\zeta_1,\zeta_2)=\log\sup_{x\in \reals}\left|\dfrac{\zeta_2-x}{\zeta_1-x}\right|
\end{equation}
for $\zeta_1,\zeta_2\in\mathbb{H}$
if we identify the Teichm\"{u}ller space of a torus $T$ with the hyperbolic plane by choosing a generator system $a, b$ of $\pi_1(T)$, and consider normalized flat structures on $T$ such that $a$ has length $1$.
In \cite{MOP}, we showed that this weak distance is indeed a Finsler metric and that it is also given by  the formula
\begin{equation}
\label{eq:torus_case_finsler-randers}
ds_{hyp}+\dfrac{1}{2}d\log {\rm Im}(\zeta),
\end{equation}
where $ds_{hyp}$ is the hyperbolic metric on $\mathbb{H}$ of constant curvature $-4$. 
Since the Teichm\"{u}ller metric coincides with the hyperbolic metric in this setting, the weak distance in \cref{eq:BPT} is nothing but the Teichm\"{u}ller--Randers metric given by \cref{eq:torus_case_finsler-randers}, that is, associated with the 1-form $\omega=-(1/2)d\log {\rm Im}(\zeta)^{-1}$. For $0\le t\le 1$, we define $\delta_t$ be the weak metric defined by the Finsler norm $ds_{hyp}+\dfrac{t}{2}d\log {\rm Im}(\zeta)$.
We also note that the 1-form $\omega=-(1/2)d\log {\rm Im}(\zeta)^{-1}$ is exact and $\rm Im (\zeta)^{-1}$ coincides with the extremal length of the isotopy class of simple closed curves corresponding to $a$.


We now turn to stating our main theorems.
Before that, we recall that the Teichm\"uller distance is a uniquely geodesic metric, and that any geodesic extends to a holomorphic disc called a Teichm\"{u}ller disc.
Namely, for any two points in $\teich_{g,m}$, there is a holomorphic (or anti-holomorphic) isometry $(\mathbb{H},d_{hyp}) \to (\teich_{g,m},d_T)$ whose image contains the two points, and this image is called a Teichm\"{u}ller disc.
A Teichm\"uller disc is determined by a holomorphic quadratic differential $q$, hence we denote it by $\mathbb{D}_q$ (see \cref{subsec:Proof_thm:extremal_length}). 

For a measured foliation $F$ on $\Sigma_{g,m}$, we  denote by $\ext_x(F)$ the function on  $\teich_{g,m}$ taking a point $x$ to the extremal length of a measured foliation $F$ at $x$, and by $q_{F,x}$ the Hubbard--Masur differential on $x$ for $F$ (see  \cref{subsec:Teichmuller-theory}). We shall show the following three main theorems.

\begin{theorem}[Geodesics of the Teichm\"{u}ller--Randers metric]
\label{thm:extremal_length}
Let $F$ be a measured foliation on $\Sigma_{g,m}$, and set $\displaystyle \omega=-\frac{1}{2}d\,\log \ext_{(\cdot)}(F)$.   
Then the following hold.
\begin{enumerate}[(i)]
\item
The (asymmetric) metric space $(\teich_{g,m},\delta^{t\omega}_T)$ is a uniquely geodesic space such that the Teichm\"uller geodesics are the geodesics. 
\item
For any $x\in \teich_{g,m}$ and for any $0\leq t\leq 1$, the Teichm\"uller disc defined by $q_{F,x}$ coincides with the image of an isometric embedding of $(\mathbb{H},\delta_t)$ into  $(\teich_{g,m},\delta^{t\omega}_T)$.  
\end{enumerate}
\end{theorem}

\begin{theorem}[Isometric discs]
  \label{thm:isometry}
Suppose that there is an isometry $\phi\colon (\mathbb{H},\delta)\to (\teich_{g,m}, \delta_T^{\omega})$ where $\omega$ is exact and satisfies $\|\omega\|_T\le 1$ in a neighbourhood of the image of $\phi$. Then, there is a measured foliation $F$ on $\Sigma_{g,m}$ such that $\phi$ is a holomorphic or anti-holomorphic isometry onto the Teichm\"uller disc associated with $q_{F,x}$  with $x=\phi(i)$, and
such that $
\omega=-(1/2)\,d\log \ext_{(\cdot)}(F)
$
holds on the image of that isometry.
\end{theorem}

%
%

From \cref{thm:extremal_length} and \cref{thm:isometry}, for a fixed measured foliation $F$, we have a characteristic property of the geometry of the weak distance $\delta^\omega_T$ with $\omega=-(1/2)\,d\log \ext_{(\cdot)}(F)$ on the Teichm\"uller disc defined by the Hubbard--Masur differential for $F$. Since, by \cref{thm:extremal_length}, any Teichm\"uller disc is totally geodesic with respect to $\delta^\omega_T$, it is natural to ask how the weak distance $\delta^\omega_T$ behaves on Teichm\"uller discs other than the one associated with $q_{F,x}$. 

\begin{theorem}
\label{thm:other_disc}
Let $F$ be a measured foliation on $\Sigma_{g,m}$, and set $\omega=-(1/2)\,d\log \ext_{(\cdot)}(F)$.For any $x \in \teich_{g,m}$ and for any measured foliation $G$ on $\Sigma_{g,m}$, we have the following.
\begin{itemize}
\item[(1)]
If $q_{G,x}$ is not a complex constant multiple of $q_{F,x}$, then  the restriction of $\delta^\omega_T$ to the Teichm\"uller disc $\mathbb{D}_{q_{G,x}}$ is a weak non-negative distance function which separates any two points.
\item[(2)]
The following two conditions are equivalent:
\begin{itemize}
\item
$i(F,G)\ne 0$.
\item
The Teichm\"uller geodesic ray directed by $q_{G,x}$ has bounded length with respect to $\delta^\omega_T$.
\end{itemize}

\end{itemize}
In particular,  the restriction of $\delta^\omega_T$ to every Teichm\"{u}ller disc is incomplete.
\end{theorem}
We note that the weak distance $\delta$ on $\mathbb{H}$, which corresponds to the Teichm\"{u}ller space of a torus, does not separate points. \Cref{thm:other_disc} gives a generalisation of this torus case. See \cref{subsec:weak_metric} for more details.

As can been seen in the definition, our metric depends on the choice of the (projective class) of a measured foliation $F$.
Because of this,  our metric is not invariant under the entire mapping class group.
Still if we consider the family of metrics making $x$ and $F$ vary, then the family is invariant under the action of the mapping class group.

%
%

Besides the theorems stated above, we shall also discuss the extension of the Hamilton-Krushkal condition (see \cref{thm:extension_Hamilton_Krushkal_condition}), and the Teichm\"{u}ller--Randers cometric on the cotangent space (see \cref{thm:Randers--Teichmuller-cometric}).

\section{Preliminaries}
\label{Preliminaries}

\subsection{Weak metric}
\label{subsec:weak_metric}
A \emph{weak metric} $\delta$ on a set $X$ is a map $\delta: X\times X\to \reals$ satisfying the following.
\begin{enumerate}
\item  $\delta(x,x)= 0$ for every $x$ in $X$;
\item  $\delta(x,y)\geq 0$ for every $x$ and $y$  in $X$;
\item $\delta(x,y)+\delta(y,z)\geq \delta(x,z)$ for every $x$, $y$ and $z$ in $X$.
\end{enumerate}
A weak metric $\delta$ is said to \emph{separate points} if $\delta(x_1,x_2)=0$ for $x_1,x_2\in X$ implies $x_1=x_2$, and
to be \emph{complete} if for any sequence $(x_n)$ in $X$ satisfying $\delta(x_n,x_{n+m})\to 0$ as $n,m\to \infty$, the sequence $(x_n)$  converges in $X$ (see \cite[I.1]{MR0296877}). (Notice that since the metric is not symmetric, the order of the arguments in $\delta(x_1,x_2)$ is important.)

In  \cite{BPT}, the following weak metric was introduced on $\hyperbolic$.
First, for $\zeta_1\not= \zeta_2 \in \hyperbolic$, we set 
\begin{eqnarray}
\label{M}
M(\zeta_1, \zeta_2)=\sup_{x\in \reals}\left|\dfrac{\zeta_2-x}{\zeta_1-x}\right|.
\end{eqnarray}
For $\zeta_1= \zeta_2$, we set $M(\zeta_1, \zeta_2)=1$. We set
$$
\delta(\zeta_1,\zeta_2)=\log M(\zeta_1,\zeta_2)\quad (\zeta_1,\zeta_2\in \mathbb{H}).
$$
Then, $\delta$ is an asymmetric weak metric on $\mathbb{H}$. Furthermore, $\delta$ does not separate points of $\mathbb{H}$. Indeed, when $\zeta_1=y_1i$, $\zeta_2=y_2i\in \mathbb{H}$ with $y_1<y_2$, 
$\delta(\zeta_1,\zeta_2)=0$. In particular, $\delta$ is not complete (see \cite[Proposition 1]{BPT}). The distance between $\zeta_1$ and $\zeta_2\in \mathbb{H}$ is explicitly given by
$$
\delta(\zeta_1,\zeta_2)=\log \left(\dfrac{|\zeta_2-\overline{\zeta_1}|+|\zeta_2-\zeta_1|}{|\zeta_1-\overline{\zeta_1}|}
\right)
$$
(see \cite{BPT}).
Hence, any geodesic ray tending to $\mathbb{R}\subset \partial \mathbb{H}$ has bounded length, and the length of a geodesic ray is infinite only if it goes upward in the vertical direction. 

We note that when we identify $\mathbb{H}$ with the Teichm\"uller space of a torus, the ideal boundary $\partial \mathbb{H}$ is naturally thought of as the Thurston boundary, which is the space of projective measured foliations on the torus (see \cite{MR568308}). Using this identification, we see that the intersection number $i(F_{x},F_{\infty})$ is zero if and only if $x=\infty$, where $[F_x]$ is the projective measured foliation corresponding to an arbitrary $x\in \partial \mathbb{H}$.
This corresponds to the condition of \cref{thm:other_disc} in the case of torus.

\subsection{Teichm\"uller theory}
\label{subsec:Teichmuller-theory}
We review some of Teichm\"uller theory. We refer the reader to \cite{MR1215481} for more details.
\subsubsection{Teichm\"uller space}
Let $\Sigma_{g,m}$ be a closed orientable surface of  type $(g,m)$, that is, of genus $g$ with $m$ points deleted The integers $g$ and $m$ may take all nonnegative values except that if $g=0$ we assume $m\geq 4$ and if $g=1$ we assume that $m\geq 1$. A \emph{marked Riemann surface} $(X,f)$ of type $(g,m)$ is a pair of an analytically finite Riemann surface $X$ of type $(g,m)$ and an orientation-preserving homeomorphism $f\colon \Sigma_{g,m}\to X$. Two marked Riemann surfaces $(X_1,f_1)$ and $(X_2,f_2)$ are said to be \emph{Teichm\"uller equivalent} if there is a conformal mapping $h\colon X_1\to X_2$ such that $h\circ f_1$ is homotopic to $f_2$. The set $\teich_{g,m}$ of Teichm\"uller equivalence classes of marked Riemann surfaces of type $(g,m)$ is called the \emph{Teichm\"uller space} of analytically finite Riemann surfaces of type $(g,m)$. 
The \emph{Teichm\"uller distance} $d_T$ on $\teich_{g,m}$ is defined by
$$
d_T(x,y)=\dfrac{1}{2}\log \inf_h K(h)
$$
where $h$ ranges over all quasi-conformal maps $h\colon X_1\to X_2$ homotopic to $f_2\circ f_1^{-1}$ and where $K(h)$ denotes the maximal quasiconformal dilatation of $h$. 
The Teichm\"uller space is known to be a complex manifold which is biholomorphically equivalent to a bounded domain in $\mathbb{C}^{3g-3+m}$. Furthermore, the Teichm\"uller distance is complete, uniquely geodesic, and coincides with the Kobayashi distance.

\subsubsection{Infinitesimal theory}
For a Riemann surface $X$, let $L^\infty(X)$ be the complex Banach space of bounded measurable $(-1,1)$-forms $\mu=\mu(z)(d\overline{z}/dz)$ on $X$ with the norm
$$
\|\mu\|_\infty={\rm ess.sup}\{|\mu(z)|\mid z\in X\}.
$$
A form  in $L^\infty(X)$ is called a \emph{Beltrami differential}.

Let $A^2(X)$ be the Banach space of holomorphic quadratic differentials $\varphi=\varphi(z)dz^2$ with the norm
$$
\|\varphi\|_1=\int_X|\varphi(z)|dxdy.
$$
There is a natural pairing between Beltrami differentials and holomorphic quadratic differentials as follows.
\begin{equation}
\label{eq:pairing_1}
L^\infty(X)\times A^2(X)\ni (\mu,\varphi)\mapsto 
\int_X\mu \varphi.
\end{equation}
Let $N^\infty(X)\subset L^\infty(X)$ be the subspace orthogonal to $A^2(X)$ with respect to the pairing. Namely,
$$
N^\infty(X):=\left\{\mu\in L^\infty(X)\mid \int_X\mu \varphi=0, \ \forall \varphi\in A^2(X)\right\}.
$$
Two Beltrami differentials $\mu$ and $\nu$ are \emph{infinitesimally Teichm\"uller equivalent} if $\mu-\nu\in N^\infty(X)$.
The (holomorphic) tangent space $T_x\teich_{g,m}$ at $x=(X,f)\in \teich_{g,m}$ is canonically identified with the quotient space $L^\infty(X)/N^\infty(X)$. 
Hence, for $x=(X,f)\in \teich_{g,m}$, the non-degenerate pairing \cref{eq:pairing_1} descends to a pairing
$$
T_x\teich_{g,m}\times A^2(X)\ni (v,\varphi)\mapsto \langle v,\varphi\rangle=\int_X\mu \varphi
$$
where $v=[\mu]$ with $\mu\in L^\infty(X)$. From this observation, the space $A^2(X)$ is canonically identified with the (holomorphic) cotangent space $T^*_x\teich_{g,m}$ at $x\in \teich_{g,m}$. A real $1$-form $\omega$ is \emph{presented by} $\psi\in A^2(X)$ at $x=(X,f)\in \teich_{g,m}$ if
$$
\omega(v)={\rm Re}\langle v,\psi\rangle
$$
for all $v\in T_x\teich_{g,m}$. Notice that in general any real $1$-form on a complex manifold is the real part of a $(1,0)$-form.

Teichm\"{u}ller proved in \cite[Sec 25]{Teich, TeichTr} that  a Finsler distance is defined by a metric, called the \emph{Teichm\"uller metric}, defined by
$$
\teichmullernorm{x}{v}=\sup\left\{
{\rm Re}
\langle v, \varphi\rangle\mid \varphi\in A^2(X), \|\varphi\|_1=1
\right\}
$$
for $v\in T_x\teich_{g,m}$. 
This distance coincides with the Teichm\"{u}ller distance defined above.

We call a Beltrami differential $\mu$ on $X$ a \emph{Teichm\"uller differential} when it has a form $\mu=c\overline{\varphi}/|\varphi$ for some $\varphi\in A^2(X)-\{0\}$.
It is known that when a Beltrami differential $\mu$ is \emph{infinitesimally extremal} in the sense that $\|\mu\|_\infty={\rm Re}
\langle v, \varphi\rangle$, then $\mu$ must be a Teichm\"uller differential.
Fix a basepoint $x_0=(X_0,f_0)$ in $\teich_{g,m}$. For $t\ge 0$ and $\varphi\in A^2(X_0)$, we denote by $F_t\colon X_0\to X_t=F_t(X_0)$ a quasi-conformal map with the property that $\overline{\partial}F_t=\tanh(t)(\overline{\varphi}/|\varphi|)\partial F_t$. Then, a path
$$
r_\varphi\colon [0,\infty)\ni t\mapsto (X_t,F_t\circ f_0)\in \teich_{g,m}
$$
contitutes a geodesic ray with respect to $d_T$. We call such a geodesic the \emph{Teichm\"uller geodesic ray} emanating from $x_0$. It is known that for any $x\in \teich_{g,m}-\{x_0\}$, there is a unique Teichm\"uller geodesic ray passing $x$ and emanating from $x_0$. Furthermore, 
$$
(0,\infty)\times \{\varphi\in A^2(X_0)\mid \|\varphi\|=1\}
\ni (t,\varphi)\mapsto r_\varphi(t)\in \teich_{g,m}\setminus \{x_0\}
$$
is a homeomorphism.

Unless $(g,m)$ is either $(1,1)$ or $(0,4)$, the Teichm\"uller metric is not Riemannian. 
This was known to Teichm\"{u}ller, but we can prove it just by using the fact that the group of linear isometries of a tangent (or cotangent) space of any Riemannan metric is a orthogonal group, whereas by \cite{\cite{MR1322950}, \cite{MR348098},
MR0288254}}, the linear isometry group of a tangent/cotangent space with respect to the Teichm\"{u}ller metric is a finite union of $S^1$.

The following might be well known and follows from the discussion in the proof of Lemma 3 in \cite[p. 173]{MR903027}. For completeness, we give a brief proof.

\begin{lemma}[Derivative of the Teichm\"uller norm]
\label{lem:derivative_Teichmuller}
Take $x=(X,f)\in \teich_{g,m}$ and $v_0\in T_x\teich_{g,m}-\{0\}$. Suppose that $v_0$ is represented by the Teichm\"uller differential $\beta\overline{\alpha_0}/|\alpha_0|$ ($\|\alpha_0\|=1$, $\beta>0$). Then
$$
\left.\dfrac{d}{dt}\right|_{t=0}\teichmullernorm{x}{v_0+tv}={\rm Re}\langle v,\alpha_0\rangle
$$
for any $v\in T_x\teich_{g,m}$.
\end{lemma}

\begin{proof}
For $t\in \mathbb{R}$, we take $\alpha_t\in A^2(X)$ with $\|\alpha_t\|_1=1$ such that 
$$
\teichmullernorm{x}{v_0+tv}={\rm Re}\langle v_0+tv,\alpha_t\rangle.
$$
We claim that $\mathbb{R}\ni t\mapsto \alpha_t\in A^2(X)$ is well defined  and continuous. Indeed, by the Lebesgue dominated convergence theorem, the map
\begin{equation}
\label{eq:linear}
L_1\colon A^2(X)\ni \alpha\mapsto \ell_\alpha=\left[A^2(X)\ni \varphi\mapsto \|\alpha\|_1\int_X\dfrac{\overline{\alpha}}{|\alpha|}\varphi\right]\in A^2(X)^*
\end{equation}
is continuous with respect to the weak topology on $A^2(X)^*$. Since $A^2(X)$ is finite-dimensional, the weak topology on $A^2(X)^*$ coincides with the topology derived from the dual norm (the operator norm). One can see that ${\rm Re}(\ell_\alpha(\psi))\le \|\alpha\|_1\|\psi\|_1$ for all $\psi\in A^2(X)$, and ${\rm Re}(\ell_\alpha(\varphi))=\|\alpha\|_1\|\varphi\|_1$ if and only if $\varphi=\alpha$. Hence, $\alpha_t$ is well defined for $t\in \mathbb{R}$, and the map defined in  \cref{eq:linear} is injective and proper. Since $A^2(X)$ and $A^2(X)^*$ are homeomorphic to the Euclidean space $\mathbb{R}^{6g-6+2m}$, from the invariance of domains, \cref{eq:linear} is homemorphic.
Since
\begin{equation}
\label{eq:linear2}
L_2\colon T_x\teich_{g,m}\ni v\mapsto [\varphi\mapsto \langle v,\varphi\rangle]\in A^2(X)^*
\end{equation}
is a complex linear isomorphism, and continuous with respect to the Teichm\"uller metric and the dual norm, we see that the map $\mathbb{R}\ni t\mapsto \alpha_t=L_1^{-1}\circ L_2(v_0+tv)\in A^2(X)$ is continuous.

Then,
\begin{align*}
\teichmullernorm{x}{v_0+tv}-\teichmullernorm{x}{v_0}
&={\rm Re}\langle v_0+tv,\alpha_t\rangle-{\rm Re}\langle v_0,\alpha_0\rangle \\
&\ge{\rm Re}\langle v_0+tv,\alpha_0\rangle-{\rm Re}\langle v_0,\alpha_0\rangle
=t\,{\rm Re}\langle v,\alpha_0\rangle
\end{align*}
and
\begin{align*}
\teichmullernorm{x}{v_0+tv}-\teichmullernorm{x}{v_0}
&\le {\rm Re}\langle v_0+tv,\alpha_t\rangle-{\rm Re}\langle v_0,\alpha_t\rangle \\
&=t{\rm Re}\langle v,\alpha_0\rangle
+t{\rm Re}\langle v,\alpha_t-\alpha_0\rangle \\
&=t\,{\rm Re}\langle v,\alpha_0\rangle
+o(t)
\end{align*}
as $t\to 0$. Therefore,
$$
|\teichmullernorm{x}{v_0+tv}-\teichmullernorm{x}{v_0}-t\,{\rm Re}\langle v,\alpha_0\rangle|=o(t)
$$
as $t\to 0$.
\end{proof}
\subsubsection{Measured foliations}
Let $\mathcal{S}$ be the set of homotopy classes of non-trivial and non-peripheral simple closed curves on $\Sigma_{g,m}$. Let $\mathcal{WS}$ be the set of formal scalar products $\{t\alpha\mid t\ge 0, \alpha\in \mathcal{S}\}$, which we call the set of \emph{weighted simple closed curves} on $\Sigma_{g,m}$. Consider the embedding
$$
\mathcal{WS}\ni t\alpha\mapsto [\mathcal{S}\ni \beta\mapsto t\,i(\alpha,\beta)]\in \mathbb{R}_{\ge 0}^{\mathcal{S}}.
$$
We equip the function space $\mathbb{R}_{\ge 0}^{\mathcal{S}}$ with the pointwise convergence topology.
The closure $\mathcal{MF}$ of the image of $\mathcal{WS}$ in $\mathbb{R}_{\ge 0}^{\mathcal{S}}$ is called the space of \emph{measured foliations} on $\Sigma_{g,m}$. 
For $F\in \mathcal{MF}$, we call the value $F(\alpha)$ the \emph{intersection number} of $F$ with $\alpha$, and denote it by $i(F,\alpha)$. Set $i(F,t\alpha)=t\,i(F,\alpha)$ for $t\alpha\in \mathcal{WS}$. It is known that the intersection number $i(\,\cdot\, ,\,\cdot\,)$ on  $\mathcal{MF}\times \mathcal{WS}$ extends continuously to a function $i(\,\cdot\, ,\,\cdot\,)$  on $\mathcal{MF}\times \mathcal{MF}$ which satisfies $i(F,G)=i(G,F)$ for $F,G\in \mathcal{MF}$.

\subsubsection{Hubbard--Masur differentials and Extremal length}
Let $x=(X,f)$ be a point in $\teich_{g,m}$. For $q=q(z)dz^2\in A^2(X)$, we set 
$$
v(q)(\alpha)=\inf_{\alpha'\in f(\alpha)}\int_{\alpha'}|{\rm Re}(\sqrt{q(z)}dz)|
$$
for $\alpha\in \mathcal{S}$.  Regarding $v(q)$ as contained in $\mathbb{R}_{\ge 0}^{\mathcal{S}}$, we call it the \emph{vertical foliation} of $q$. It is known that $v(q)\in \mathcal{MF}$.

For $x=(X,f)\in \teich_{g,m}$ and $F\in \mathcal{MF}$, there is a unique quadratic differential $q_{F,x}\in A^2(X)$ such that $i(F,\alpha)=v(q)(\alpha)$ for all $\alpha\in \mathcal{S}$. We call the differential $q_{F,x}$ the \emph{Hubbard--Masur differential} for $F$ on $x$. The norm
$$
\ext_x(F)=\int_{X}|q_{F,x}(z)|dxdy
$$
is called the \emph{extremal length} of $F$ on $x$. The extremal length function
$$
\teich_{g,m}\times \mathcal{MF}\ni (x,F)\mapsto \ext_x(F)
$$
is continuous.
When $F\in \mathcal{MF}$ is fixed, the extremal length function is of class $C^1$. The following formula, called the \emph{Gardiner formula},  is known:
\begin{equation}
\label{eq:Gardiner-formula}
d\,\ext_{\cdot}(F)=-2{\rm Re}\int_X\mu q_{F,x}
\end{equation}
for $v=[\mu]\in T_x\teich_{g,m}\cong L^\infty(X)/N^\infty(X)$ (cf. \cite{MR736212}). Notice that the minus sign in the right-hand side of \cref{eq:Gardiner-formula} comes from the fact that $q_{F,x}$ has $F$ as the vertical foliation, while Gardiner considers the horizontal foliations when he concludes the formula \cref{eq:Gardiner-formula}.

\subsection{Teichm\"uller--Randers metric}
For a given $n$-dimensional Riemannian manifold $(M,g)$ and a 1-form $\omega$ on $M$ with $\|\omega\|_g<1$ at every point of $M$, the associated \emph{Randers metric} is a Finsler metric on $M$ defined by $F(v)=g(v,v)^{1/2}+\omega(v)$.
Although in general literatures, Randers metrics refer to deformations of Riemannian metrics by 1-forms,  we can think of such deformations for Finsler (symmetric) metrics in the same way.
Furthermore, even in the case when $\|\omega\|_g=1$, the Randers metric makes sense as a weak Finsler metric.
In this paper, we study Randers-type deformations of the Teichm\"{u}ller metric $\kappa$ on $\teich_{g,m}$ which we explained in the previous subsection, by a 1-form $\omega$ as we presented in Introduction ;
%
\begin{equation}
\randersteichmullernorm{x}{v}{\omega}
=\teichmullernorm{x}{v}+\omega(v).
\end{equation}
As a $1$-form $\omega$, we shall consider in particular the form expressed as $-\frac{1}{2}d\,\log \ext_{(\cdot)}(F)$ for a measured foliation $F$ on $\Sigma_{g,m}$.
This metric does depends on the choice of $F$, but only on the projective class of $F$ since we are taking $\log$ in the second term.
This metric can be regarded as a generalisation of the weak Finsler metric which we studied in \cite{MOP}.
\subsection{References to  background materials}
We now give  some references for background materials which we briefly explained in this section. 
The Teichm\"uller metric was introduced and studied thoroughly by Teichm\"uller in his paper \cite{Teich} (see its English translation \cite{TeichTr}). 
In this paper there is a long section (\S 25) on the Finsler nature of this metric. 
In the same section, Teichm\"uller introduced and studied what are now called Teichm\"uller discs (isometric images of the hyperbolic plane, defined by quadratic differentials), which he calls complex geodesics. 
As modern introductions to Teichm\"uller theory, we refer the to \cite{MR903027} and \cite{MR1215481}. 
For the theory of measured foliations and measured foliation spaces, we refer the reader to \cite{MR568308}, and for a comprehensive introduction to extremal length \cite{MR3289702} for instance.
For Randers' metric, we refer the reader to Rander's original paper \cite{MR927291}.

\section{Extension of the Hamilton--Krushkal condition}

In this section, we discuss the infinitesimal extremal property for our Teichm\"{u}ller--Randers metric.

Let $X$ be a Riemann surface, and fix $\varphi_0\in A^2(X)$.
We consider the following functional on the space $L^\infty(X)$ of Beltrami differenitals:
\begin{equation}
\label{eq:functional}
\beta(\mu,\varphi_0)
=\sup\left\{
\left|\int_X\mu \varphi\right|+{\rm Re}\int_X\mu\varphi_0\mid \varphi\in A^2(X), \|\varphi\|_1=1.
\right\}
\end{equation}
for $\mu \in L^\infty(S)$.
It immediately follows from the definition that
\begin{equation}
\label{eq:beta_mu}
\beta(\mu,\varphi_0)\le \|\mu\|_\infty+{\rm Re}\int_X\mu\varphi_0
\end{equation}
for all $\mu\in L^\infty(X)$.
We say that a Beltrami differential  $\mu$ is \emph{infinitesimally $\varphi_0$-extremal} if
$$
\beta(\mu,\varphi_0)=\|\mu\|_\infty+{\rm Re}\int_X\mu\varphi_0,
$$
 and that $\mu$ satisfies the \emph{Hamilton condition} if 
$$
\sup\left\{\left|\int_X\mu\varphi\right|\mid \varphi\in A^2(X), \|\varphi\|_1\le 1\right\}=\|\mu\|_\infty.
$$
It is known that $\mu\in L^\infty(X)$
is \emph{infinitesimally Teichm\"uller extremal} in the sense that $\|\mu-\nu\|_\infty\ge \|\mu\|_\infty$ for all $\nu\in N^\infty(X)$ if and only if it satisfies the Hamilton condition (\cite{MR245787} and \cite{MR0241633}).
%
%
%

\begin{theorem}[Extension of the Hamilton--Krushkal condition]
\label{thm:extension_Hamilton_Krushkal_condition}
Let $X$ be a Riemann surface and $\varphi_0$ a holomorphic quadratic differential on $X$.
\begin{itemize}
\item[{\rm (1)}]
If two Beltrami differentials $\mu,\nu\in L^\infty(X)$ are infinitesimally Teichm\"uller equivalent, then $\beta(\mu,\varphi_0)=\beta(\nu,\varphi_0)$.
\item[{\rm (2)}] For a Beltrami differential $\mu\in L^\infty(X)$, the following three conditions are equivalent:
\begin{itemize}
\item[{\rm (a)}]
$\mu$ is \emph{infinitesimally $\varphi_0$-extremal};
\item[{\rm (b)}]
$\mu$ is \emph{infinitesimally Teichm\"uller extremal}; and
\item[{\rm (c)}]
$\mu$ satisfies the Hamilton condition.
\end{itemize}
%
\end{itemize}
\end{theorem}

\begin{proof}
(1)\quad The assumption that $\mu$ and $\nu$ are Teichm\"{u}ller equivalent means by definition that $\int_X\mu\varphi=\int_X\nu\varphi$ for all $\varphi\in A^2(X)$. 
Hence, we have
$$
\left|\int_X\mu \varphi\right|+{\rm Re}\int_X\mu\varphi_0
=
\left|\int_X\nu \varphi\right|+{\rm Re}\int_X\nu\varphi_0
$$
for all $\varphi\in A^2(X)$, which implies $\beta(\mu,\varphi_0)=\beta(\nu,\varphi_0)$.

\medskip
\noindent
(2)\quad We only need to show the equivalence between conditions (a) and (b), for the equivalence of the condition (c) with (a) and (b)  follows immediately then.
Suppose that $\mu$ is infinitesimally $\varphi_0$-extremal. Then for $\nu\in N^\infty(X)$, we have
\begin{align*}
&\|\mu\|_\infty+{\rm Re}\int_X\mu\varphi_0
=\beta(\mu,\varphi_0) =\beta(\mu-\nu,\varphi_0) \\
&\le \|\mu-\nu\|_\infty+{\rm Re}\int_X(\mu-\nu)\varphi_0 = \|\mu-\nu\|_\infty+{\rm Re}\int_X\mu\varphi_0,
\end{align*}
and hence $\|\mu\|_\infty\le \|\mu-\nu\|_\infty$. This means that $\mu$ is infinitesimally Teichm\"uller extremal.

Conversely, suppose that $\mu$ is infinitesimally Teichm\"uller extremal. Then, by definition, there exists a sequence $(\varphi_n)$ in $A^2(X)$ such that $|\int_X\mu\varphi_n|\to \|\mu\|_\infty$. Therefore,
$$
\|\mu\|_\infty+{\rm Re}\int_X\mu\varphi_0
=\lim_{n\to \infty}\left|\int_X\mu\varphi_n\right|+{\rm Re}\int_X\mu\varphi_0
\le \beta(\mu,\varphi_0).
$$
Combining this with  \cref{eq:beta_mu}, we see that $\mu$ is infinitesimally $\varphi_0$-extremal.
%
\end{proof}
%
In the case of analytically finite Riemann surfaces, an infinitesimally Teichm\"uller extremal Beltrami differential is a Teichm\"uller Beltrami differential, and vice versa. Hence we have the following.

\begin{corollary}[Analytically finite case]
\label{coro:teichmuller_differenital}
Let $X$ be an analytically finite Riemann surface and $\varphi_0$ a holomorphic quadratic differential on $X$.
Then, for $\mu\in L^\infty(X)$, the following two conditions are equivalent.
\begin{itemize}
\item[{\rm (a)}]
$\mu$ is infinitesimally $\varphi_0$-extremal;  and
\item[{\rm (b)}]
$\mu$ is a Teichm\"uller Beltrami differential. Namely, there are $\psi\in A^2(X)-\{0\}$ and $c\ge 0$ such that $\mu=c\overline{\psi}/|\psi|$.
\end{itemize}
\end{corollary}

%
%

\subsection{Teichm\"{u}ller--Randers cometric}
Let $x=(X,f)$ be a point in $\teich_{g,m}$, and $\omega$ a $1$-form on $\teich_{g,m}$ with $\Vert \omega\Vert_T(x) \leq 1$. 
We define the \emph{Teichm\"uller--Randers cometric}  on the space of holomorphic quadratic differentials, which identified with the cotangent space as $A^2(X)\cong T^*_x\teich_{g,m}$, by
$$
G_\omega(\varphi)=\sup_{\randersteichmullernorm{x}{v}{\omega}=1}|\langle v,\varphi\rangle|=\sup_{\randersteichmullernorm{x}{v}{\omega}=1}{\rm Re}\langle v,\varphi\rangle
$$
for $\varphi\in A^2(X)$.
This is dual to the Randers --Teichm\"uller metric. When $\omega=0$, it is known that
$$
G_0(\varphi)=\|\varphi\|_1.
$$
Even if $\|\omega\|_T(x)<1$, as we have seen above, $G_\omega$ defines an (asymmetric) norm on $A^2(X)$, whereas $G_\omega$ is not a norm then. 
Indeed, take a tangent vector $v\in T_x\teich_{g,m}$ with $\teichmullernorm{x}{v}=1$ and $\omega(v)=-1$. Then, $\randersteichmullernorm{x}{v}{\omega}=0$ by definition of $\kappa^\omega$. Take $\alpha\in A^2(X)$ such that ${\rm Re}\langle v,\alpha\rangle=\|\alpha\|_1=1$, and $\{v_n\}\subset T_x\teich_{g,m}$ converging to $v$. Then, we have
$$
\dfrac{{\rm Re}\langle v_n,\alpha\rangle}{\randersteichmullernorm{x}{v_n}{\omega}}\to \infty
$$
as $n\to \infty$, which shows that $G_\omega$ is not a norm. For this reason, when we discuss the dual $G_\omega$, we always assume that $\|\omega\|_T(x)< 1$.


\begin{theorem}[Teichm\"uller--Randers cometric]
\label{thm:Randers--Teichmuller-cometric}
Let $x=(X,f)$ be a point in $teich_{g,m}$. Suppose that $\omega$ is represented by $\psi\in A^2(X)\cong T^*_x\teich_{g,m}$ at $x$ and 
 that $\|\omega\|_T(x)=\|\psi\|_1<1$. Then,
\begin{equation}
\label{eq:cometric}
G_\omega(\varphi)=\inf\left\{t>0 \mid \left\|\frac{\varphi}{t}-\psi\right\|_1\le 1\right\}.
\end{equation}
In particular, if $\varphi\ne 0$, then we have $$\left\|\frac{\varphi}{G_\omega(\varphi)}-\psi\right\|_1=1.$$
\end{theorem}


\begin{proof}
By  \cref{coro:teichmuller_differenital}, for a Teichm\"uller Beltrami differential $\mu=\overline{\alpha}/|\alpha|$ ($\alpha\in A^2(X)$),  we have
\begin{equation}
\label{eq:randers_teichmuller_beltrami}
\randersteichmullernorm{x}{[\mu]}{\omega}=1+{\rm Re}\int_X\dfrac{\overline{\alpha}}{|\alpha|}\psi.
\end{equation}
We may assume that neither $\varphi$ nor $\psi$ is $0$, for our claim evidently holds if one of them is $0$.
%
%
From the definition of $G_\omega$ and \cref{eq:randers_teichmuller_beltrami}, we have
\begin{equation}
\label{eq:G_omega_2}
G_\omega(\varphi)=\sup_{\teichmullernorm{x}{v}=1}\dfrac{{\rm Re}\langle v,\varphi\rangle}{1+{\rm Re}\langle v,\psi\rangle}.
\end{equation}
If $v$ is represented by  the Teichm\"uller Beltrami differential $\overline{\varphi}/{|\varphi|}$, 
the function in the supremum in the right-hand side of \cref{eq:G_omega_2} is positive. Hence we have $G_\omega(\varphi)>0$.

Let $v_0$ be a tangent vector represented by a  Teichm\"uller Beltrami differential $\overline{\alpha_0}/|\alpha_0|$ ($\|\alpha_0\|_1=1$) which attains the supremum in the right-hand side of \cref{eq:G_omega_2}.
Then
$$
G_\omega(\varphi)=\dfrac{{\rm Re}\langle v_0,\varphi\rangle}{1+{\rm Re}\langle v_0,\psi\rangle}.
$$
We note that ${\rm Re}\langle v_0,\varphi\rangle> 0$ and $|\langle v_0,\psi\rangle|<1$ since $G_\omega(\varphi)> 0$ and $\|\psi\|_1<1$. 
For $v\in T_x\teich_{g,m}$, we can compute
\begin{equation*}
\label{eq:proof_G_omega_1}
\dfrac{\displaystyle{\rm Re}\langle v_0+tv,\varphi\rangle}{\displaystyle1+{\rm Re}\langle v_0+tv,\psi\rangle}
=G_\omega(\varphi)+tG_\omega(\varphi){\rm Re}\left\langle v,
\dfrac{\varphi}{{\rm Re}\langle v_0,\varphi\rangle}-\dfrac{\psi}{1+{\rm Re}\langle v_0,\psi\rangle}\right\rangle+o(t)
\end{equation*}
as $t\to 0$.
As remarked above, $G_\omega(\varphi) >0$.
Since the left hand side of the above equality attains the supremum in $\{v\mid \kappa(x;v)=1\}$ at $v_0$, by \cref{lem:derivative_Teichmuller},
we have
$$
{\rm Re}\left\langle v,
\dfrac{\varphi}{{\rm Re}\langle v_0,\varphi\rangle}-\dfrac{\psi}{1+{\rm Re}\langle v_0,\psi\rangle}\right\rangle=0
$$
for all $v\in T_x\teich_{g,m}$ with ${\rm Re}\langle v,\alpha_0\rangle=0$. 
This means that there exists  $t\in \mathbb{R}$ such that
$$
\dfrac{\varphi}{{\rm Re}\langle v_0,\varphi\rangle}-\dfrac{\psi}{1+{\rm Re}\langle v_0,\psi\rangle}=t\alpha_0.
$$
By taking  pairing with $v_0$ on both sides, we get
$$
\dfrac{{\rm Re}\langle v_0,\varphi\rangle}{{\rm Re}\langle v_0,\varphi\rangle}-\dfrac{{\rm Re}\langle v_0,\psi\rangle}{1+{\rm Re}\langle v_0,\psi\rangle}=t\|\alpha_0\|_1=t,
$$
which means  $t=(1+{\rm Re}\langle v_0,\psi\rangle)^{-1}$. Thus we obtain
$$
\dfrac{\varphi}{G_\omega(\varphi)}-\psi=
\dfrac{1+{\rm Re}\langle v_0,\psi\rangle}{{\rm Re}\langle v_0,\varphi\rangle}\varphi-\psi=\alpha_0,
$$
which implies the desired equalities.
\end{proof}

\subsection{The case when $\omega$ is exact}
Assume that $\omega$ is a continuous exact form, 
that is, $\omega=dF_\omega$ for some $C^1$-function $F_\omega$ on $\teich_{g,m}$.
Then, the length of any $C^1$-path $\gamma\colon [a,b]\to \teich_{g,m}$ is expressed as
\begin{align*}
\int_a^b\randersteichmullernorm{\gamma(t)}{\dot{\gamma}(t)}{\omega}dt
&=\int_a^b(\teichmullernorm{\gamma(t)}{\dot{\gamma}(t)}+\omega(\dot{\gamma}(t)))
dt \\
&=\int_a^b\teichmullernorm{\gamma(t)}{\dot{\gamma}(t)}dt+F_\omega(\gamma(b))-F_\omega(\gamma(a))
\end{align*}
with respect to the Teichm\"{u}ller--Randers metric $\kappa^\omega$.
Therefore, taking the infimum on the lengths of paths connecting $x_1\in \teich_{g,m}$ to $x_2\in \teich_{g,m}$, 
the weak metric $\delta^\omega_T$ associated with the Teichm\"{u}ller--Randers metric
satisfies
\begin{equation}
\label{exact case}
\delta^\omega_T(x_1,x_2)=d_T(x_1,x_2)+F_\omega(x_2)-F_\omega(x_1).
\end{equation}
This inequality implies the following proposition.

\begin{proposition}
\label{prop:geodesic}
For any continuous exact form $\omega$ on $\teich_{g,m}$,
any Teichm\"uller geodesic is a unique geodesic with respect to the Teichm\"{u}ller--Randers distance $\delta^\omega_T$.
\end{proposition}

This gives a generalisation of Theorem 2.1 in \cite{MOP}.
We also note that in the case when $\omega$ is exact, the symmetrisation of the weak metric associated with the Teichm\"{u}ller--Randers metric coincides with the Teichm\"uller distance. Indeed, we have
\begin{align*}
&S(\delta^\omega_T)(x_1,x_2)
=\dfrac{1}{2}(\delta^\omega_T(x_1,x_2)+\delta^\omega_T(x_2,x_1)) \\
&=\dfrac{1}{2}(d_T(x_1,x_2)+(F_\omega(x_2)-F_\omega(x_1))+
d_T(x_2,x_1)+(F_\omega(x_1)-F_\omega(x_2)))\\
&=\dfrac{1}{2}(d_T(x_1,x_2)+
d_T(x_2,x_1))=d_T(x_1,x_2).
\end{align*}

%

\begin{proof}[Proof of \cref{prop:geodesic}]
By \cref{exact case}, we see immediately that every Teichm\"{u}ller geodesic is also a geodesic with respect to $\delta_T^\omega$.
It remains to check the uniqueness of geodesics.
Let $x_1,x_2\in \teich_{g,m}$ and $\gamma\colon[a,b]\to \teich_{g,m}$ be a $C^1$-path connecting $x_1$ to $x_2$.
If $\gamma$ is not a Teichm\"uller geodesic, by the uniqueness of Teichm\"{u}ller geodesics, we have
\begin{align*}
&d_T(x_1,x_2)+F_\omega(x_2)-F_\omega(x_1) \\
&<\int_a^b(\teichmullernorm{\gamma(t)}{\dot{\gamma}(t)}dt
+F_\omega(x_2)-F_\omega(x_1)
=\int_a^b\randersteichmullernorm{\gamma(t)}{\dot{\gamma}(t)}{\omega}dt,
\end{align*}
which implies that $\gamma$ is not a geodesic with respect to $\delta_T^\omega$ either.
\end{proof}

%
%
%
%
%

\section{Proof of theorems}
\subsection{Teichm\"uller discs}
Let $x=(X,f)$ be a point in $\teich_{g,m}$.
For $q\in A^2(X)$ and $\lambda\in \mathbb{H}$, we define
\begin{equation}
\label{eq:mu-q-lambda}
\mu_{\lambda,q}:=\dfrac{\lambda-i}{\lambda+i}\dfrac{\overline{q}}{|q|}.
\end{equation}
Let $f_{\lambda,q}$ be a quasi-conformal map on $X$ with $\bar \partial f_{\lambda, q}=\mu_{\lambda, q} \partial f_{\lambda, q}$, and set $X_{\lambda,q}$ to be the image of $f_{\lambda,q}$.
The \emph{Teichm\"uller disc} associated with $q$, which is denoted by $\mathbb D_q$, is a holomorphic disc in $\teich_{g,m}$ defined by
\begin{equation}
\label{eq:Teichmuller-discs}
\phi_q\colon\mathbb{H}\ni\lambda\mapsto x(\lambda,q):=(X_{\lambda,q},f_{\lambda,q})\in\teich_{g,m}.
\end{equation}

The following lemma shows basic properties of Teichm\"{u}ller discs.
\begin{lemma}
\label{lem:extremal_length_Teich}
For $x=(X,f)\in\teich_{g,m}$ and  a measured foliation $F$ on $S$, we have the following.
\begin{itemize}
\item[{\rm (a)}]
 The extremal length function satisfies
\begin{equation}
\label{eq:extremal_length_disc}
\ext_{x(\lambda,q_{F,x})}(F)=\dfrac{1}{{\rm Im}(\lambda)}\ext_x(F).
\end{equation}
\item[{\rm (b)}]
For the Teichm\"uller disc defined as in \cref{eq:Teichmuller-discs} for $q=q_{F,x}$, the image of any vertical geodesic line in $\mathbb{H}$ is the  Teichm\"uller geodesic defined by  holomorphic quadratic differentials whose vertical foliations are $F$.
\item[{\rm (c)}]
For any measured foliation $F$ on $S$ and any $\lambda\in\mathbb{H}$, the unit tangent vector $v_\lambda=(\phi_{q_{F,x}})_*\left(2i{\rm Im}(\lambda)\,\partial/\partial\lambda\right)$ is represented by $\overline{q_{F,\phi(\lambda)}}/|q_{F,\phi(\lambda)}|$.
\end{itemize}
\end{lemma}

\begin{proof}
The assertions follow from the discussion by Marden and Masur in \cite[\S1.3]{MR393584}. 
We review some details for the convenience of the reader. 

\noindent
{\rm (a)}\quad
We shall only show \cref{eq:extremal_length_disc} for $\alpha\in \mathcal{S}$. 
Since the weighted simple closed curves are dense in $\mathcal{MF}$  and $\mathcal{MF}\ni F\mapsto q_{F,x}\in A^2(X)$ is continuous, we can then conclude \cref{eq:extremal_length_disc} for general measured foliations   by taking limits.

One of the characterisations of the extremal length of $\alpha$ is that it is the reciprocal of the modulus of the \lq characteristic annulus' of $q_\alpha$, that is, the maximal (open) annulus formed by closed leaves of the vertical trajectories of $q_\alpha$. (See also \cite[\S20.3]{MR743423}).
By the discussion by Marden and Masur in \cite[\S1.3]{MR393584}, the extremal length $\ext_{x(\lambda,q_{F,x})}(\alpha)$ satisfies
\begin{equation}
\label{Marden-Masur}
\ext_{x(\lambda,q_{F,x})}(\alpha)=\dfrac{1}{1+{\rm Re}(\lambda')}\ext_{x}(\alpha)
\end{equation}
where $\lambda'$ is a complex number satisfying ${\rm Re}(\lambda')>-1$ and
\begin{equation*}
\dfrac{\lambda-i}{\lambda+i}=\dfrac{\lambda'}{2+\lambda'}.
\end{equation*}
Since ${\rm Re}(\lambda')={\rm Re}\left(-1-i\lambda\right)=-1+{\rm Im}(\lambda)$,  we obtain \cref{eq:extremal_length_disc} from \cref{Marden-Masur} in the case when $F=\alpha\in \mathcal{S}$. 

\medskip
\noindent
(b)\quad
Let $A$ be the characteristic annulus of $q_{\alpha,x}$.
The Teichm\"uller map $f_{\lambda, q}$ defined by $\mu_{\lambda,q_{\alpha,x}}$ is expressed as a map $h_\lambda$ defined by
$$
h_\lambda(z)=z|z|^{-i\lambda-1}=z|z|^{{\rm Im}(\lambda)-1-i{\rm Re}(\lambda)}
$$
on the characteristic annulus $A\cong \{1<|z|<r\}$ with $r=\exp(2\pi/\ext_{x}(\alpha))$. 
The image $h_\lambda(\{1<|z|<r\})=\{1<|z|<r^{{\rm Im}(\lambda)}\}$ corresponds to the characteristic annulus of the terminal quadratic differential $q_{\alpha,x(\lambda,q_{\alpha,x})}$. 
Therefore, the deformation along the vertical line in $\mathbb{H}$ passing through $\lambda\in \mathbb{H}$ is the Teichm\"uller geodesic associated with the differential $q_{\alpha,x(\lambda, q_{\alpha, x})}$.

\medskip
\noindent
{\bf (c)}\quad 
Let $v_\lambda\in T_{x(\lambda,q_{F,x})}\teich_{g,m}$ be the unit tangent vector in $\mathbb{D}_{q_{F,x}}$ at $x(\lambda,q_{F,x})$ as given in the statement (c). 
Then $v_\lambda$ is represented by a Teichm\"uller Beltrami differential $\overline{\psi}/|\psi|$ with $\psi\in A^2(X_{\lambda,q_{F,x}})$.
From (a) above and the Gardiner formula \cref{eq:Gardiner-formula},
\begin{align*}
-{\rm Re}\int_{X_{\lambda,q_{F,x}}}\dfrac{\overline{\psi}}{|\psi|}
\dfrac{q_{F,x(\lambda,q_{F,x})}}{\|q_{F,x(\lambda,q_{F,x})}\|_1}
&=
\dfrac{1}{2}d\log \ext_{x(\cdot,q_{F,x})}(F)[v_\lambda]\\
&=
{\rm Re}\left(2i{\rm Im}(\lambda)\dfrac{d}{d\lambda}\log \ext_{x(\cdot,q_{F,x})}(F)\right) \\
&={\rm Re}\left(2i{\rm Im}(\lambda)\cdot \left(-\dfrac{1}{2i{\rm Im}(\lambda)}\right)\right)\\
&
=-1,
\end{align*}
which means that $\psi=q_{F,x(\lambda,q_{F,x})}$.
\end{proof}
\subsection{Proof of \cref{thm:extremal_length}}
\label{subsec:Proof_thm:extremal_length}
The part (i) follows from  \cref{prop:geodesic}.
Let $x$ be a point in $\teich_{g,m}$ and let $\phi\colon \mathbb{H}\to \teich_{g,m}$ be the Teichm\"uller disc defined by $q_{F,x}$ with $\phi(i)=x$. Since $\omega$ is exact, by \cref{prop:geodesic},
for any two points $\zeta_1,\zeta_2\in \mathbb{H}$,
the hyperbolic geodesic $\gamma\colon [a,b]\to \mathbb{H}$ connecting $\zeta_1$ to $\zeta_2$ is mapped to a geodesic with respect to $\delta^\omega_T$ connecting $x_1=\phi(\zeta_1)$ to $x_2=\phi(\zeta_2)$.
From \cref{eq:torus_case_finsler-randers}
and \cref{lem:extremal_length_Teich}, we have
\begin{align*}
\delta_T^\omega(x_1,x_2)
&=\int_a^b(\teichmullernorm{\phi(\gamma(s))}{\phi_*\circ \dot{\gamma}(t)}+\omega(\phi_*\circ \dot{\gamma}(t)))dt \\
&=d_{hyp}(\zeta_1,\zeta_2)-\dfrac{1}{2}\int_{\phi(\gamma)}
d\log\ext_\cdot(F)
\\
&=d_{hyp}(\zeta_1,\zeta_2)+\dfrac{1}{2}\log \ext_{x_1}(F)-\dfrac{1}{2}\log \ext_{x_2}(F)
\\
&=d_{hyp}(\zeta_1,\zeta_2)+\dfrac{1}{2}\log \dfrac{1}{{\rm Im}(\zeta_1)}-\dfrac{1}{2}\log \dfrac{1}{{\rm Im}(\zeta_2)} \\
&=\int_\gamma\left(ds_{hyp}+\dfrac{1}{2}d\log {\rm Im}(\zeta)\right)=\delta(\zeta_1,\zeta_2),
\end{align*}
which implies the part (ii) of \cref{thm:extremal_length}.

\subsection{Proof of \cref{thm:isometry}}
Let $\phi \colon \mathbb H \to \teich_{g,m}$ be an isometry as in the statement.
We may assume that $\omega$ is exact on $\teich_{g,m}$ by changing it outside a neighbourhood of $\phi(\mathbb H)$. 
Then, there is a $C^1$-function $F_\omega$ on $\teich_{g,m}$ such that $dF_\omega=\omega$. 
We set $f_\omega=F_\omega\circ \phi$.

Take two points $\zeta_1=\xi_1+i\eta_1$ and
$\zeta_2=\xi_2+i\eta_2\in \mathbb{H}$.
Let $\gamma\colon [0,s_0]\to \mathbb{H}$ be a hyperbolic geodesic connecting $\zeta_1$ to $\zeta_2$. 
By \cref{prop:geodesic}, $\phi\circ\gamma\colon [0,s_0]\to \teich_{g,m}$ is a Teichm\"uller geodesic, and since $\phi$ is an isometry, we obtain
\begin{align}
d_{hyp}(\zeta_1,\zeta_2)+\dfrac{1}{2}\log \dfrac{\eta_2}{\eta_1}
&=\delta(\zeta_1,\zeta_2)=\delta^\omega_T(\phi(\zeta_1),\phi(\zeta_2))
\nonumber
\\
&=\int_{0}^{s_0}(\teichmullernorm{\phi(\gamma(t))}{\phi_*\circ \dot{\gamma}(t)}+\omega(\phi_*\circ \dot{\gamma}(t)))dt
\nonumber \\
&=
d_T(\phi(\zeta_1),\phi(\zeta_2))+
f_\omega(\zeta_2)-f_\omega(\zeta_1).
\label{eq:isometry1}
\end{align}

\medskip
\noindent
{\bf Case 1.}(horizontal lines)\quad
Suppose that $\eta_1=\eta_2$.
Since both $d_{hyp}$ and $d_T$ are symmetric,
from \cref{eq:isometry1}, we obtain
\begin{equation}
\label{eq:f_omega_horizontal1}
f_\omega(\zeta_1)=f_\omega(\zeta_2),\quad \text{and hence}\quad
d_T(\phi(\zeta_1),\phi(\zeta_2))=d_{hyp}(\zeta_1,\zeta_2).
\end{equation}

\medskip
\noindent
{\bf Case 2.}(vertical lines)\quad
Suppose $\xi_1=\xi_2$ and $\eta_1>\eta_2$. In this case, the geodesic $\gamma$ is a vertical segment from $\zeta_1$ to $\zeta_2$. 
Since $\delta(\zeta_1, \zeta_2)=0$ in this case, from \cref{eq:isometry1}, we have
\begin{equation}
\label{eq:delta_integrate-3}
f_\omega(\zeta_1)-f_\omega(\zeta_2)
=d_T(\phi(\zeta_1),\phi(\zeta_2)).
\end{equation}
For $x\in \mathbb{R}$, let $L_\xi=\{\xi+\eta i\mid \eta>0\}$. 
Then by \cref{eq:delta_integrate-3}, we see that $f(L_\xi)$ is a geodesic with respect to $d_T$.
Take a measured foliation $F_\xi$ on $S$ such that $\phi(L_\xi)$ is the Teichm\"uller geodesic defined by the Hubbard--Masur differential for $F_\xi$. 
To describe this more precisely, fix $\eta_0>0$ and set $x(\xi)=\phi(\xi+i\eta_0) \in \teich_{g,m}$. 
Let $x(\xi+i\eta)$ be the image of the Teichm\"uller map from $x(\xi)$ with the Betrami differenital
\begin{equation}
\label{eq:teichmuller-Beltrami-L_xi}
\tanh(t)\dfrac{\overline{q_{F_\xi,x(\xi)}}}{|q_{F_\xi,x(\xi)}|},
\end{equation}
where $t=t(\eta)$ satisfies $|t|=d_T(x(\xi+i\eta),x(\xi))$, $t>0$ if $\eta>\eta_0$, and $t\le 0$ otherwise.
Then we have
$\phi(L_\xi)=\left\{x(\xi+i\eta)\mid \eta>0\right\}$.
By the Gardiner formula \cref{eq:Gardiner-formula}, $\ext_{x(\xi+i\eta)}(F_\xi)$ decreases as $\eta$ increases.
Hence, from the Kerckhoff formula, we have
$$
\ext_{x(\xi+i\eta)}(F_\xi)=
\begin{cases}
e^{-2d_T(x(\xi+i\eta),x(\xi))}\ext_{x(\xi)}(F_\xi) & (\eta\ge \eta_0) \\
e^{2d_T(x(\xi+i\eta),x(\xi))}\ext_{x(\xi)}(F_\xi) & (\eta\le \eta_0).
\end{cases}
$$

Now, for any $\eta, \eta'$, take $\eta_3>0$  smaller than $\min\{\eta, \eta'\}$.
From \cref{eq:delta_integrate-3}, we can compute as follows:
\begin{align}
\int_{\gamma'} \omega&=f_\omega(\xi+i\eta')-f_\omega(\xi+i\eta) \label{eq:f_omega_vertical}
\\
&=
(f_\omega(\xi+i\eta_2)-f_\omega(\xi+i\eta_3))-(f_\omega(\xi+i\eta_1)-f_\omega(\xi+i\eta_3))
\nonumber \\
&=
\dfrac{1}{2}\log\dfrac{\ext_{x(\xi+i\eta_3)}(F_\xi)}{\ext_{x(\xi+i\eta')}(F_\xi)}
-\dfrac{1}{2}\log\dfrac{\ext_{x(\xi+i\eta_3)}(F_\xi)}{\ext_{x(\xi+i\eta)}(F_\xi)}
\nonumber \\
&=
-\dfrac{1}{2}\log\dfrac{\ext_{x(\xi+i\eta')}(F_\xi)}{\ext_{x(\xi+i\eta)}(F_\xi)}
\nonumber \\
&=-\dfrac{1}{2}\int_{\gamma'}d\log\ext_{\cdot}(F_\xi),
\nonumber
\end{align}
where $\gamma'$ is the image  under $\phi$ of the vertical segment from $\xi+i\eta$ to $\xi+i\eta'$ in $\mathbb{H}$.
We note that by \cref{eq:teichmuller-Beltrami-L_xi} or (c) of \cref{lem:extremal_length_Teich}, the tangent vector $v_y\in T_y\teich_{g,m}$ along $\phi(L_\xi)$ at $y\in \phi(L_\xi)$ has unit length with respect to the Teichm\"uller metric, and is given by
the Beltrami differential
\begin{equation}
\label{eq:Beltrami-L_xi}
\dfrac{\overline{q_{F_\xi,x}}}{|q_{F_\xi,x}|}.
\end{equation}
Hence we obtain
\begin{align*}
\left|-\dfrac{1}{2}d\log \ext_{\cdot}(F_\xi)[v_y]\right|
=\dfrac{1}{2}\cdot \dfrac{2}{\|q_{F,x}\|}{\rm Re}\int_X\dfrac{\overline{q_{F_\xi,x}}}{|q_{F_\xi,x}|}q_{F_\xi,x}=1.
\end{align*}
Since $\|\omega\|_T\le 1$ on the image $\phi(\mathbb{H})$, from \cref{eq:f_omega_vertical}, we conclude that we have
\begin{equation}
\label{eq:omega_on_vertical_segment}
\omega=-\dfrac{1}{2}d\log \ext_{\cdot}(F_\xi).
\end{equation}
on $L_\xi$. 

\medskip
\noindent
{\bf Case 3. }(general case)\quad
We take $\zeta_1=\xi_1+i\eta_1$ and
$\zeta_2=\xi_2+i\eta_2\in \mathbb{H}$ to be arbitrary. 
Set  $\zeta_3=\xi_2+i\eta_1$.
By \cref{eq:extremal_length_disc} we have
\begin{equation}
\label{eq:extremal_length_on_L_xi}
\ext_{\phi(\xi+i\eta_1)}(F_\xi)=\dfrac{\eta_2}{\eta_1}\ext_{\phi(\xi+i\eta_2)}(F_\xi)
\end{equation}
for all $\xi\in \mathbb{R}$ and $\eta_1,\eta_2>0$, and $\ext_{\phi(\zeta_3)}(F_{\xi_2})=\ext_{\phi(\zeta_1)}(F_{\xi_1})$. 
Combining this with the argument in Case 2, we have
\begin{align*}
f_\omega(\zeta_2)-f_\omega(\zeta_1)
&=(f_{\omega}(\zeta_2)-f_\omega(\zeta_3))+(f_\omega(\zeta_3)-f_\omega(\zeta_1))\\
&=-\dfrac{1}{2}\log \ext_{\phi(\zeta_2)}(F_{\xi_2})+\dfrac{1}{2}\log \ext_{\phi(\zeta_1)}(F_{\xi_2})\noindent \\
&=-\dfrac{1}{2}\log \dfrac{\eta_1}{\eta_2}.
\end{align*}
Then, from \cref{eq:isometry1}, we conclude that $d_T(\phi(\zeta_1),\phi(\zeta_2))=d_{hyp}(\zeta_1,\zeta_2)$ for any $\zeta_1$, $\zeta_2\in \mathbb{H}$. Hence, $\phi\colon (\mathbb{H},d_{hyp})\to (\teich_{g,m},d_T)$ is an isometry. 
\cite[Theorem 1.1]{MR3608291} shows that in this situation, $\phi$ is either holomorphic or anti-holomorphic, and the image is the Teichm\"uller disc. 
As shown in \cref{lem:extremal_length_Teich}, $F_{\xi_1}=F_{\xi_2}$ for all $\xi_1,\xi_2\in \mathbb{R}$. 
Setting $F=F_\xi$ ($\xi\in\mathbb{R})$, we see that the image $\phi(\mathbb{H})$ is the Teichm\"uller disc defined by the Hubbard--Masur differential for $F$.

Consider $\zeta=\xi+i\eta \in \mathbb{H}$ and $L_\xi$ defined above. 
By \cref{eq:f_omega_horizontal1}, the derivative of $f_\omega$ at $\zeta$ in the horizontal direction is constantly zero. 
As shown in (c) of \cref{lem:extremal_length_Teich}, the image $v\in T_{\phi(\zeta)}\teich_{g,m}$ of the unit tangent vector $2i{\rm Im}(\zeta)(\partial/\partial\zeta)\in T_\zeta\mathbb{H}$ to $L_\xi$ at $\zeta$ is represented by the Teichm\"uller Beltrami differential
$$
\dfrac{\overline{q_{F,\phi(\zeta)}}}{|q_{F,\phi(\zeta)}|}.
$$
Hence, by the Gardiner formula \cref{eq:Gardiner-formula}, the derivative of the function
$$
\mathbb{H}\ni \zeta\mapsto -\dfrac{1}{2}\log \ext_{\phi(\zeta)}(F)
$$
is also zero in the horizontal direction in $\mathbb{H}$. As a consequence, by \cref{eq:omega_on_vertical_segment},
$$
\omega=-\dfrac{1}{2}d\log\ext_{\cdot}(F)
$$
on the image $\phi(\mathbb{H})$.

\subsection{Proof of \cref{thm:other_disc}}
Let $x$ be a point in $\teich_{g,m}$,  and $G$ a measured foliation on $S$. 

\noindent
(1)  Suppose that  $\alpha q_{F,x}\ne q_{G,x}$ for any complex number $\alpha$, and hence $\mathbb{D}_{q_{F,x}}\cap \mathbb{D}_{q_{G,x}}=\{x\}$.

Then we claim the following.
\begin{claim}
\label{claim:1}
For any $y\in \mathbb{D}_{q_{G,x}}$, we have $\mathbb{D}_{q_{F,y}}\cap \mathbb{D}_{q_{G,y}}=\{y\}$.
\end{claim}

\begin{proof}
Otherwise, there is $y\in D_{q_{G,x}}$ such that $\mathbb{D}_{q_{F,y}}$ and $\mathbb{D}_{q_{G,y}}$ share at least two points. By the uniqueness of the Teichm\"uller geodesic, $\mathbb{D}_{q_{F,y}}$ and $\mathbb{D}_{q_{G,y}}$ share a common Teichm\"uller geodesic line passing through these two points. Since  $\mathbb{D}_{q_{F,y}}$ and $\mathbb{D}_{q_{G,y}}$ are holomorphic discs,  by the identity theorem, $\mathbb{D}_{q_{F,y}}=\mathbb{D}_{q_{G,y}}$. 
Since both $x$ and $y$ lie in $\mathbb{D}_{q_{G,x}}=\mathbb{D}_{q_{F,y}}$, from the discussion in \cref{lem:extremal_length_Teich} (or the discussion in \cite[\S1.3]{MR393584}), we have $\mathbb{D}_{q_{F,y}}=\mathbb{D}_{q_{F,x}}$ and $\mathbb{D}_{q_{G,x}}=\mathbb{D}_{q_{G,y}}$ . Therefore, we obtain
$$
\mathbb{D}_{q_{G,x}}=\mathbb{D}_{q_{G,y}}=\mathbb{D}_{q_{F,y}}=\mathbb{D}_{q_{F,x}},
$$
which contradicts our assumption.
\end{proof}

Let $y$ be a point in $\mathbb{D}_{q_{G,x}}$, and $v_y$ the unit tangent vector to $\mathbb{D}_{q_{G,x}}$ at $y$ represented by
$\overline{q_{G,y}}/|q_{G,y}|$ (cf. (c) of \cref{lem:extremal_length_Teich}). 
We note that by \cref{claim:1}, $q_{G,y}$ is not a complex scalar multiple of $q_{F,y}$.
Hence,
\begin{align*}
\left|-\dfrac{1}{2}d\log \ext_{\cdot}(F)[v_y]\right|
&=\left|
{\rm Re}\int_{X_{\lambda,q_{G,y}}}\dfrac{\overline{q_{G,y}}}{|q_{G,y}|}
\dfrac{q_{F,x(\lambda,q_{F,y})}}{\|q_{F,y}\|_1}
\right|<1.
\end{align*}
Therefore, for any compact set $K$ in $\mathbb{D}_{q_{G,x}}$, 
there is a constant $C_K<1$ such that
$$
\left|-\dfrac{1}{2}d\log \ext_{\cdot}(F)[v_y]\right|\le C_K
$$
for all $y\in K$.

Let $x_1$ and  $x_2$ be distinct points on $\mathbb{D}_{q_{G,x}}$, and $\gamma\subset \mathbb{D}_{q_{G,x}}$ the Teichm\"uller geodesic containing  $x_1$ and $x_2$. From the above discussion, we have
$$
\left|
\dfrac{1}{2}\log \ext_{x_1}(F)-\dfrac{1}{2}\log \ext_{x_2}(F)
\right|
<d_T(x_1,x_2)
$$
and
$$
\delta^\omega_T(x_1,x_2)=d_T(x_1,x_2)+\dfrac{1}{2}\log \ext_{x_1}(F)-\dfrac{1}{2}\log \ext_{x_2}(F)>0,
$$
which implies that $\delta^\omega_T$ separates two points in $\mathbb{D}_{q_{G,x}}$.

\noindent
(2)
Let $r_G=r_{G,x}\colon [0,\infty)\to \teich_{g,m}$ be the Teichm\"uller geodesic ray defined by $q_{G,x}$ with arclength parameterisation. 
By \cite[Lemma 1]{MR3009545}, the function
$$
[0,\infty)\ni t\mapsto 
e^{-\delta^{\omega}_T(x,r_G(t))}=e^{-t}\left(\dfrac{\ext_{r_G(t)}(F)}{\ext_{x_0}(F)}\right)^{1/2}
$$
is non-increasing and tends to $\dfrac{\mathcal{E}(F)}{\ext_{x_0}(F)^{1/2}}$ as $t\to \infty$ where $\mathcal{E}$ is some continuous function defined on $\mathcal{MF}$ (see also \cite[Theorem 1.1]{MR2449148}). 
Let $G=G_1+\cdots+G_m$ be the decomposition of $G$ into indecomposable components (for detail, see \cite{MR2449148}).  In \cite[Corollary 1]{MR3956189}, Walsh showed that the limit function $\mathcal{E}$ is expressed as
$$
\mathcal{E}(H)=\sqrt{
\sum_{i=1}^m\dfrac{i(G_i,H)^2}{i(G_i,\mathcal{H}(q_{G,x}))}}
$$
for $H\in\mathcal{MF}$,
where $\mathcal{H}(q_{G,x})$ is the horizontal foliation of $q_{G,x}$. Therefore, $\mathcal{E}(H)=0$ if and only if $i(G,H)=0$.
This means that $\delta^\omega_T(x,r_G(t))$ is uniformly bounded in terms of $t\ge 0$ if and only if $i(F,G)\ne 0$.

Finally, we prove the incompleteness of the restriction of $\delta_T^\omega$ to any Teichm\"uller disc. 
Let $x$ be a point in $\teich_{g,m}$ and $G$ a measured foliation on $S$. By \cite[Theorem 2]{MR855297}, the vertical foliation $G_\theta$ of $e^{i\theta}q_{G,x}$ is uniquely ergodic for almost every $\theta$. Therefore,  $i(F,G_\theta)\ne 0$ for almost every $\theta$.
It follows that almost all Teichm\"uller geodesic rays emanating from $x$ in $\mathbb{D}_{q_{G,x}}$ have bounded length with respect to the distance $\delta^\omega_T$, and in particular, the restriction of $\delta^\omega_T$ to $\mathbb{D}_{q_{G,x}}$ is incomplete.

\bibliographystyle{acm}
\bibliography{mop.bib}

\end{document}